\documentclass[11pt,leqno]{amsart}
\usepackage{verbatim,amssymb,amsfonts,latexsym,amsmath,amsthm,bbm,nicefrac,xspace,enumerate,tikz}

\newtheorem{theorem}{Theorem}[section]

\newtheorem{lemma}[theorem]{Lemma}
\newtheorem{corollary}[theorem]{Corollary}

\newtheorem{question}[theorem]{Question}

\theoremstyle{definition}

\theoremstyle{remark}

\newcommand{\Q}{\mathbb{Q}}
\newcommand{\R}{\mathbb{R}}
\newcommand{\N}{\mathbb{N}}

\newcommand{\CC}{\mathbb{C}}

\newcommand{\A}{\mathcal{A}}
\newcommand{\B}{\mathcal{B}}
\newcommand{\C}{\mathcal{C}}

\renewcommand{\P}{\mathcal{P}}

\newcommand{\explicitSet}[1]{\left\lbrace #1 \right\rbrace}
\newcommand{\brackets}[1]{\left\langle #1 \right\rangle}
\newcommand{\set}[2]{\explicitSet{#1 \colon #2}}
\newcommand{\seq}[2]{\brackets{#1 \colon #2}}
\newcommand{\<}{\langle}
\renewcommand{\>}{\rangle}
\renewcommand{\a}{\alpha}
\renewcommand{\b}{\beta}
\newcommand{\g}{\gamma}
\newcommand{\dlt}{\delta}

\newcommand{\z}{\zeta}
\renewcommand{\k}{\kappa}
\newcommand{\s}{\sigma}
\renewcommand{\t}{\tau}
\newcommand{\w}{\omega}
\newcommand{\0}{\emptyset}

\newcommand{\sub}{\subseteq}
\newcommand{\rest}{\!\restriction\!}
\newcommand{\cat}{\!\,^{\frown}}

\newcommand{\homeo}{\approx}

\newcommand{\card}[1]{\left\lvert #1 \right\rvert}
\newcommand{\tr}[1]{[\![#1]\!]}
\newcommand{\PP}{\mathbb{P}}
\newcommand{\forces}{\Vdash}

\newcommand{\TT}{\mathbb{T}}
\newcommand{\continuum}{\mathfrak{c}}

\newcommand{\dom}{\mathfrak d}

\newcommand{\mad}{\mathfrak{a}}

\newcommand{\ch}{\ensuremath{\mathsf{CH}}\xspace}
\newcommand{\gch}{\ensuremath{\mathsf{GCH}}\xspace}

\newcommand{\ser}{\acute{\mathfrak{n}}}
\newcommand{\bspec}{\mathfrak{sp}(\text{\small Borel})}
\newcommand{\Spec}{\mathfrak{sp}(\text{\small closed})}
\newcommand{\specod}{\mathfrak{sp}(\text{\footnotesize OD($\R$)})}
\newcommand{\specg}{\mathfrak{sp}({\text{\footnotesize $\Gamma$}})}

\begin{document}

\title[Partitioning the real line into Borel sets]{Partitioning the real line into Borel sets}
\author{Will Brian}
\address {
W. R. Brian\\
Department of Mathematics and Statistics\\
University of North Carolina at Charlotte\\
9201 University City Blvd.\\
Charlotte, NC 28223, USA}
\email{wbrian.math@gmail.com}
\urladdr{wrbrian.wordpress.com}



\begin{abstract}
For what infinite cardinals $\kappa$ is there a partition of the real line $\mathbb R$ into precisely $\kappa$ Borel sets? 

Hausdorff famously proved that there is a partition of $\mathbb R$ into $\aleph_1$ Borel sets. 
But other than this, we show that the spectrum of possible sizes of partitions of $\R$ into Borel sets can be fairly arbitrary. 
For example, given any $A \subseteq \omega$ with $0,1 \in A$, there is a forcing extension in which $A = \{ n :\, \text{there is a partition of }\mathbb R\text{ into }\aleph_n\text{ Borel sets}\}$.

We also look at the corresponding question for partitions of $\R$ into closed sets. 
We show that, like with partitions into Borel sets, the set of all uncountable $\kappa$ such that there is a partition of $\mathbb R$ into precisely $\kappa$ closed sets can be fairly arbitrary. 
\end{abstract}

\maketitle


\section{Introduction}

Hausdorff showed in \cite{Hausdorff} that the Cantor space $2^\w$, and in fact any uncountable Polish space, can be expressed as an increasing union $\bigcup_{\xi < \w_1}E_\xi$ of $G_\dlt$ sets. 
It follows that $\R$, or any other uncountable Polish space, can be partitioned into $\aleph_1$ nonempty $F_{\s \dlt}$ sets.
This raises the question:
\begin{itemize}

\vspace{1mm}

\item[({\footnotesize Q1})] For which uncountable cardinals $\k$ is there a partition of $\R$ into $\k$ nonempty Borel sets?

\vspace{1mm}

\end{itemize} 
Let us define the \emph{Borel partition spectrum}, denoted $\bspec$, to be the answer to this question:
$$\bspec = \set{|\P|}{\P \text{ is a partition of } \R \text{ into uncountably many Borel sets} }\!.$$
The main result of this paper shows that the Borel partition spectrum can be fairly arbitrary. (A precise statement of the result can be found in Corollary~\ref{cor:0} below.)
For example, given any set $A$ of positive integers, there is a forcing extension in which $\bspec = \set{\aleph_n}{n \in A} \cup \{\aleph_1,\aleph_\w,\aleph_{\w+1}\}$.

A related question can be asked for any given class of subsets of $\R$. 
As $\R$ is connected, there is no partition of $\R$ into $2$ or more open sets. So the descriptively simplest sets one can ask this for are the closed sets.

\begin{itemize}

\vspace{1mm}

\item[({\footnotesize Q2})] For which cardinals $\k$ is there a partition of $\R$ into $\k$ closed sets?

\vspace{1mm}

\end{itemize}

Making use of the Baire Category Theorem, Sierpi\'nski proved in \cite{Sierpinski} that any partition of $\R$ into at least $2$ nonempty closed sets is uncountable. (Using the modern vocabulary, his proof actually shows that any partition of $\R$ into closed sets has size $\geq\!\mathrm{cov}(\mathcal M)$.)
Let us define the \emph{closed partition spectrum}, denoted $\Spec$, to be the answer to ({\footnotesize Q2}):
$$\Spec = \set{|\P|}{\P \text{ is a partition of } \R \text{ into at least 2 closed sets} }\!.$$
We show that, like the Borel partition spectrum, $\Spec$ can be fairly arbitrary -- even more so, in fact, since the closed partition spectrum need not contain $\aleph_1$. 
For example, given any set $A$ of positive integers, there is a forcing extension in which $\Spec = \set{\aleph_n}{n \in A} \cup \{\aleph_\w,\aleph_{\w+1}\}$.

These results about $\bspec$ and $\Spec$ are encompassed in a single theorem, Theorem~\ref{thm:main} below. The proof identifies a notion of forcing that adds partitions of $\R$ into closed sets having certain prescribed sizes, while avoiding partitions of $\R$ into Borel sets with other sizes. 
The second of these tasks is the more difficult. It is accomplished via an isomorphism-of-names argument, similar to the folklore proof (found, e.g., in \cite[Theorem 3.1]{Brendle1}) that, after many mutually generic Cohen reals are added to a model of \ch, there are no MAD families of size strictly between $\aleph_1$ and $\continuum$.

Let us point out that analogues of ({\footnotesize Q1}) and ({\footnotesize Q2}) have been asked and answered (or partly answered) concerning other extremal families. 
For example, consider $\mathfrak{sp}(\text{\footnotesize MAD}) = \set{\k \geq \aleph_0}{\text{there is a MAD family of size } \k}$.
Hechler showed in \cite{Hechler} that $\mathfrak{sp}(\text{\footnotesize MAD})$ can include any prescribed set of cardinals, and Blass showed in \cite{Blass} how to exclude certain cardinals from $\mathfrak{sp}(\text{\footnotesize MAD})$.
Shelah and Spinas proved the strongest results in \cite{Shelah&Spinas}, showing that $\mathfrak{sp}(\text{\footnotesize MAD})$ can be rather arbitrary, especially on the regular cardinals. 
Similarly,
the spectrum of possible sizes of maximal independent families, $\mathfrak{sp}(\text{\small mif})$,
was investigated recently by Fischer and Shelah in \cite{Fischer&Shelah,FS2}. Like with $\mathfrak{sp}(\text{\footnotesize MAD})$, they proved that $\mathfrak{sp}(\text{\small mif})$ can be fairly arbitrary, especially on the regular cardinals.
Similar work concerning maximal cofinitary groups was done by Fischer in \cite{Fischer}.
Ultimately, all these proofs share the same core idea: variations on the isomorphism-of-names argument mentioned above. In every case, the key to making this kind of argument work is to find an automorphism-rich poset that can be used to add extremal families of prescribed sizes.
One notable exception to this rule is Shelah's analysis of the set of possible sizes of ultrafilter bases in \cite{Shelah}, where he proves, from large cardinal hypotheses, that this spectrum can exhibit fairly chaotic behavior.

\section{$\mathfrak{sp}(\textrm{\small Borel})$ and $\mathfrak{sp}(\textrm{\small closed})$ do not depend on $\R$}

We begin this section by observing that the definition of $\bspec$ does not depend on $\R$, and remains unchanged when $\R$ is replaced by any other uncountable Polish space:

\begin{theorem}\label{thm:bspec}
If $X$ is any uncountable Polish space, then
$$\mathfrak{sp}(\text{\emph{\small Borel}}) = \set{\k > \aleph_0}{\text{there is a partition of } X \text{ into } \k \text{ Borel sets}}.$$
\end{theorem}
\begin{proof}
By a theorem of Kuratowski (see \cite[Theorem 15.6]{Kechris}), any two uncountable Polish spaces are Borel isomorphic: in other words, there is a bijection $f: \R \to X$ such that $A \sub \R$ is Borel if and only if $f[A]$ is Borel. Thus if $\P$ is any partition of $\R$ into Borel sets, then $\set{f[B]}{B \in \P}$ is a partition of $X$ into Borel sets, and if $\mathcal Q$ is any partition of $X$ into Borel sets, then $\set{f^{-1}[B]}{B \in \mathcal Q}$ is a partition of $\R$ into Borel sets.
\end{proof}

It turns out that the same is true for $\Spec$: if $X$ is any uncountable Polish space, then
$$\mathfrak{sp}(\text{\small closed}) = \set{\k > \aleph_0}{\text{there is a partition of } X \text{ into } \k \text{ closed sets}}.$$
A related theorem is proved by Miller in \cite[Theorem 3]{Miller}: $\R$ can be partitioned into $\aleph_1$ closed sets if and only if some uncountable Polish space can be, if and only if every uncountable Polish space can be. We wish to prove the same, but with any uncountable $\k$ in place of $\aleph_1$. Miller's proof does not readily adapt to this task, because it uses in an essential way the fact that $\aleph_1$ is the smallest uncountable cardinal. So we take a different approach. 
First we need a few lemmas.

\begin{lemma}\label{lem:ser}
For any uncountable cardinal $\k$, the following are equivalent:

\vspace{1mm}

\begin{enumerate}

\item There is a partition of $2^\w$ into $\k$ closed sets.

\vspace{1mm}

\item There is a partition of $\w^\w$ into $\k$ compact sets.

\vspace{1mm}

\item For every uncountable Polish space $X$, there is a partition of $X$ into $\k$ compact sets.

\vspace{1mm}

\item For some uncountable compact Polish space $X$, there is a partition of $X$ into $\k$ $F_\s$ sets.

\end{enumerate}
\end{lemma}
\begin{proof}
Throughout the proof, if $X \sub Y$ and $\P$ is a partition of $Y$, then $\P \rest X = \set{K \cap X}{K \in \P}$ denotes the restriction of $\P$ to $X$.
We prove that $(1) \Rightarrow (2) \Rightarrow (3) \Rightarrow (4) \Rightarrow (1)$.

$(1) \Rightarrow (2)$:
Suppose $\P$ is a partition of $2^\w$ into $\k$ closed sets.
We claim first that there is a closed $X \sub 2^\w$ such that $\P \rest X$ is a partition of $X$ into $\k$ nowhere dense closed sets (nowhere dense in $X$, that is).

To see this, we define a descending transfinite sequence of closed subsets of $2^\w$. Let $X_0 = 2^\w$, and if $\a$ is a limit ordinal, take $X_\a = \bigcap_{\xi < \a}X_\xi$. At stage $\a$, given $X_\a$, form $X_{\a+1}$ by removing any open subset of $X_\a$ contained in a single member of $\P$: $X_{\a+1} = X_\a \setminus \bigcup \set{U}{U \text{ is open in }X_\a \text{ and } \card{\vphantom{f^f}\P \rest U} = 1}$. Because $2^\w$ is second countable, there is some $\b < \w_1$ such that $X_\b = X_\g$ for all $\g \geq \b$. Let $X = X_\b$.
Clearly, $\P \rest X$ is a partition of $X$ into compact nowhere dense sets. Furthermore, at any stage $\a < \b$ of our recursion, $\set{K \in \P}{K \cap X_\a \neq K \cap X_{\a+1}}$ is countable. Thus $\set{K \in \P}{K \cap X \neq K}$ is countable. It follows that $\card{\vphantom{f^f}\P \rest X} = \card{\P} = \k$.

Now, we claim there is a subspace $Y$ of $X$ such that $Y \homeo \w^\w$ and $\P \rest Y$ is a partition of $Y$ into $\k$ compact sets.
Fix a countable basis $\B$ for $X$, and for every $U \in \B$ fix some $K_U \in \P$ such that $K_U \cap U \neq \0$. Let $Y = X \setminus \bigcup \set{K_U}{U \in \B}$. Clearly $\P \rest Y = \P \setminus \set{K_U}{U \in \B}$, so $\P$ partitions $Y$ into $\k$ compact sets. That $Y \homeo \w^\w$ follows from the Alexander-Urysohn characterization of $\w^\w$ as the unique nowhere compact, zero-dimensional Polish space without isolated points (see \cite[Theorem 7.7]{Kechris}).

$(2) \Rightarrow (3)$:
Suppose $\P$ is a partition of $\w^\w$ into $\k$ compact sets, and
let $X$ be any uncountable Polish space.
Decompose $X$ into its scattered part and perfect part: i.e., let $X = Y \cup Z$, where $Y$ is countable and $Z$ is closed in $X$ (hence still Polish) and $Z$ has no isolated points.
By \cite[Exercise 7.15]{Kechris}, there is a continuous bijection $f: \w^\w \to Z$. But then $\set{f[K]}{K \in \P} \cup \set{\{y\}}{y \in Y}$ is a partition of $X$ into $\k$ compact sets.

$(3) \Rightarrow (4)$:
This implication is obvious, since ``every uncountable Polish space'' includes some compact spaces, and compact sets are $F_\s$.

$(4) \Rightarrow (1)$:
Suppose $X$ is an uncountable compact Polish space. By \cite[Theorem 7.4]{Kechris}, there is a continuous surjection $f: 2^\w \to X$. If $\P$ is any partition of $X$ into $\k$ $F_\s$ sets, then $\mathcal Q = \set{f^{-1}[K]}{K \in \P}$ is a partition of $2^\w$ into $\k$ $F_\s$ sets.
But every $F_\s$ subset of $2^\w$ can be partitioned into countably many closed sets. (This observation is attributed to Luzin in \cite[Theorem 2]{Miller}.) Thus, by breaking up any non-closed members of $\mathcal Q$ into countably many closed pieces, we can refine $\mathcal Q$ to obtain a partition of $2^\w$ into $\k$ closed sets.
\end{proof}

Define the \emph{Sierpi\'nski cardinal} $\ser$ to be the minimum size of a cardinal satisfying the equivalent statements in Lemma~\ref{lem:ser}: that is,
\begin{align*}
\ser \,&= \min \set{|\P|}{\P \text{ is a partition of } 2^\w \text{ into uncountably many closed sets} } \\
&= \min \set{|\P|}{\P \text{ is a partition of } \w^\w \text{ into compact sets} }.
\end{align*}
Recall that the dominating number $\dom$ is equal to the smallest covering of $\w^\w$ by compact sets. Hence $\dom \leq \ser$. And clearly $\ser \leq \continuum$, because $\w^\w$ can be partitioned into singletons.
Thus we may consider $\ser$ to be a cardinal characteristic of the continuum. 
Quite a bit is known already about this cardinal. The main results are due to Stern \cite{Stern}, Miller \cite{Miller}, Newelski \cite{Newelski}, and Spinas \cite{Spinas}, who studied this cardinal implicitly without giving it a name, and Hru\v{s}ak \cite{DHM,Hrusak}, who denotes it $\mad_T$. 
The name ``Sierpi\'nski cardinal'' and the notation $\ser$ were suggested by Banakh in \cite{Banakh}. 

A set $S$ of cardinals is \emph{closed under singular limits} if for every singular cardinal $\k$, if $S \cap \k$ is unbounded in $\k$ then $\k \in S$.

\begin{lemma}\label{lem:singularlimits}
The set $S = \set{\k}{\text{there is a partition of }2^\w \text{ into }\k\text{ closed sets}}$ is closed under singular limits.
\end{lemma}
\begin{proof}
Suppose $\k$ is a singular cardinal and $S \cap \k$ is unbounded in $\k$. Let $\nu = \mathrm{cf}(\k) < \k$, and let $\seq{\a_\xi}{\xi < \nu}$ be a sequence of cardinals in $S$ increasing up to $\k$. Fix some $\lambda \in S$ with $\nu \leq \lambda < \k$. 

Let $\set{K_\xi}{\xi < \lambda}$ be a partition of $2^\w$ into closed sets. 
Observe that $\set{2^\w \times K_\xi}{\xi < \lambda}$ is a partition of $2^\w \times 2^\w \homeo 2^\w$ into $\lambda$ copies of $2^\w$.
Thus we may (and do) assume that $K_\xi \homeo 2^\w$ for each $\xi < \lambda$. 
For each $\xi < \nu$, let $\P_\xi$ be a partition of $K_\xi$ into $\a_\xi$ closed sets (using the fact that $\a_\xi \in S$). Then $\bigcup_{\xi < \nu}\P_\xi \cup \set{K_\xi}{\nu \leq \xi < \lambda}$ is a partition of $2^\w$ into $\k$ closed sets.
\end{proof}

%
\begin{theorem}\label{thm:Spec}
Let $\k$ be an uncountable cardinal. Then all six statements of the following form are equivalent:
\begin{enumerate}

\item[$\ $] Some/every uncountable Polish space can be partitioned into $\k\,\,\,\,\,\,$ compact/closed/$F_\s$ sets.

\end{enumerate}
\end{theorem}
\begin{proof}
It is clear that (every-compact)$\,\Rightarrow\,$(every-closed)$\,\Rightarrow\,$(every-$F_\s$), and that (some-compact)$\,\Rightarrow\,$(some-closed)$\,\Rightarrow\,$(some-$F_\s$). Also, Lemma~\ref{lem:ser} implies that (every-$F_\s$)$\,\Rightarrow\,$(some-compact): because if every uncountable Polish space can be partitioned into $\k$ $F_\s$ sets, then in particular some uncountable compact Polish space can be, and by Lemma~\ref{lem:ser} this implies $\w^\w$ can be partitioned into $\k$ compact sets. Thus, to prove the theorem, we need to show (some-$F_\s$)$\,\Rightarrow\,$(every-compact).

So let $X$ be some uncountable Polish space, and suppose $\P$ is a partition of $X$ into $\k$ $F_\s$ sets. If $\k = \ser$, then by definition, there is a partition of $\w^\w$ into $\ser$ compact sets, and by Lemma~\ref{lem:ser} every uncountable Polish space can be partitioned into $\k$ compact sets, and we are done. So let us suppose $\k > \ser$. We consider two cases.

For the first case, suppose $\k$ is regular. 
By the definition of $\ser$ and Lemma~\ref{lem:ser}, there is a partition $\mathcal Q$ of $X$ into $\ser$ compact sets. 
Because $|\P| = \k > \ser$ and $\k$ is regular, there is some $K \in \mathcal Q$ such that $\card{\vphantom{f^f}\P \rest K} = \k$. Thus $\P \rest K$ is a partition of $K$ into $\k$ compact subsets. But $K$ is a compact Polish space, and uncountable because $|K| \geq \card{\vphantom{f^f}\P \rest K} = \k$. Thus some compact uncountable Polish space can be partitioned into $\k$ $F_\s$ sets. Invoking Lemma~\ref{lem:ser} again, this implies every uncountable Polish space can be partitioned into $\k$ compact sets.

For the second case, suppose $\k$ is singular.
As in the first case, there is a partition $\mathcal Q$ of $X$ into $\ser$ compact sets. If $\card{\vphantom{f^f}\P \rest K} = \k$ for some $K \in \mathcal Q$, then we may argue as in the first case and conclude that every uncountable Polish space can be partitioned into $\k$ compact sets, and we are done. So let us suppose instead that $\card{\vphantom{f^f}\P \rest K} < \k$ for every $K \in \mathcal Q$.
Let $\lambda$ be any cardinal with $\ser < \lambda < \k$.
Because $|\P| = \k > \lambda > \ser = |\mathcal Q|$, there is some $K \in \mathcal Q$ such that $\card{\vphantom{f^f}\P \rest K} \geq \lambda$. Let $\mu = \card{\vphantom{f^f}\P \rest K}$, and note that $\mu < \k$ (by the third sentence of this paragraph). 
Now, as in the previous paragraph, $K$ is a compact uncountable Polish space that can be partitioned into $\mu$ $F_\s$ sets. 
By Lemma~\ref{lem:ser}, this implies $2^\w$ can be partitioned into $\mu$ closed sets.
Because $\lambda$ was an arbitrary cardinal below $\k$ and $\lambda \leq \mu < \k$, this shows that 
$S = \set{\mu}{\text{there is a partition of }2^\w \text{ into }\mu \text{ closed sets}}$
is unbounded below $\k$. By Lemma~\ref{lem:singularlimits}, $\k \in S$. By Lemma~\ref{lem:ser}, this implies every uncountable Polish space can be partitioned into $\k$ compact sets.
\end{proof}

\begin{corollary}\label{cor:Spec}
For any uncountable Polish space $X$,
\begin{align*}
\mathfrak{sp}(\text{\emph{\small closed}}) &= \set{\k > \aleph_0}{\text{there is a partition of } X \text{ into }\k \text{ compact sets}} \\ 
 &= \set{\k > \aleph_0}{\text{there is a partition of } X \text{ into }\k \text{ closed sets}} \\
  &= \set{\k > \aleph_0}{\text{there is a partition of } X \text{ into }\k \ F_\s \text{ sets}}.
\end{align*}
\end{corollary}

\begin{corollary}\label{cor:serd}
$\min \!\big( \mathfrak{sp}(\text{\emph{\small closed}}) \big) = \ser \geq \dom$.
\end{corollary}

\noindent We note that the inequality $\ser \geq \dom$ has been observed before, and can be considered folklore. It is (arguably) implicit in Miller \cite{Miller}, and was observed explicitly by Hru\v{s}ak in \cite{Hrusak} and later by Banakh in \cite{Banakh}. Anticipating the main theorem in Section 3, note that this inequality gives us an easy way of excluding an initial segment of the uncountable cardinals from $\Spec$: simply make $\dom$ big.

\begin{corollary}\label{cor:singularlimits}
$\mathfrak{sp}(\text{\emph{\small closed}})$ is closed under singular limits.
\end{corollary}
\begin{proof}
This follows from Corollary~\ref{cor:Spec} and Lemma~\ref{lem:singularlimits}.
\end{proof}

It is worth pointing out that the same result holds for $\bspec$, by a very similar argument. The analogous result also holds for the set $\mathfrak{sp}(\text{\footnotesize MAD})$ mentioned in the introduction \cite[Theorem 3.1]{Hechler}, again by a similar argument.

\begin{theorem}\label{cor:singularlimits}
$\mathfrak{sp}(\text{\emph{\small Borel}})$ is closed under singular limits.
\end{theorem}
\begin{proof}
Suppose $\k$ is a singular cardinal and $\bspec \cap \k$ is unbounded in $\k$. Let $\nu = \mathrm{cf}(\k) < \k$, and let $\seq{\a_\xi}{\xi < \nu}$ be a sequence of cardinals in $\bspec$ increasing up to $\k$. Fix some $\lambda \in \bspec$ with $\nu \leq \lambda < \k$. 

Let $\set{B_\xi}{\xi < \lambda}$ be a partition of $2^\w$ into $\lambda$ Borel sets. 
Observe that $\set{2^\w \times B_\xi}{\xi < \lambda}$ is a partition of $2^\w \times 2^\w \homeo 2^\w$ into $\lambda$ uncountable Borel sets.
Thus we may (and do) assume that $B_\xi$ is uncountable for each $\xi < \lambda$. 
Every uncountable Borel set contains a closed subspace homeomorphic to $2^\w$. For each $\xi < \nu$, fix $C_\xi \sub B_\xi$ with $C_\xi \homeo 2^\w$, and
let $\P_\xi$ be a partition of $C_\xi$ into $\a_\xi$ Borel sets. 
Then $\bigcup_{\xi < \nu}\P_\xi \cup \set{B_\xi \setminus C_\xi}{\xi < \nu} \cup \set{B_\xi}{\nu \leq \xi < \lambda}$ is a partition of $2^\w$ into $\k$ Borel sets.
\end{proof}

We end this section with some terminology regarding trees, and two open questions regarding the Sierpi\'nski cardinal $\ser$.

Recall that a \emph{subtree} of $2^{<\w}$ is a subset of $2^{<\w}$ that is closed under taking initial segments. 
A subtree $T$ of $2^{<\w}$ is \emph{pruned} if every node of $T$ has a successor in $T$.
A \emph{branch} through $T$ means a function $b \in 2^\w$ such that $b \rest n \in T$ for all $n < \w$.
If $T$ is a subtree of $2^{<\w}$, we denote by $\tr{T}$ the set of all branches through $T$. 
It is not difficult to see that $\tr{T}$ is closed in $2^\w$ for any subtree $T$ of $2^{<\w}$; and conversely, for every closed $C \sub 2^\w$ there is a subtree $T$ of $2^{<\w}$ with $C = \tr{T}$ (for details, see \cite[Chapter 2]{Kechris}). Similar terminology is used for subtrees of $\w^{<\w}$ or of $2^{<k}$.

Representing closed sets with trees in this way, Theorem~\ref{thm:Spec} states that $\k \in \Spec$ if and only if $\k > \aleph_0$ and there is a MAD family of $\k$ subtrees of $2^{<\w}$. 
Similarly, $\k \in \Spec$ if and only if there is a MAD family of $\k$ finitely branching, pruned subtrees of $\w^{<\w}$. 
These characterizations of $\Spec$ explain Hru\v{s}\'ak's notation, writing $\mad_T$ for $\min \!\big( \Spec \big)$. This relationship between $\ser$ and $\mad$ raises the following questions.

\begin{question}
Is it consistent that $\ser$ has countable cofinality?
\end{question}

\noindent The corresponding question for $\mad$ was solved by Brendle in \cite{Brendlea}, where he used Shelah's template forcing technique to obtain a model of $\mad = \aleph_\w$. It is relatively easy to make $\ser$ singular of uncountable cofinality: e.g., by adding $\aleph_{\w_1}$ Cohen reals to a model of \ch, we get a model where $\dom = \ser = \continuum = \aleph_{\w_1}$.

It is also simple to prove the consistency of $\mad < \ser$. This holds, for example, in the Cohen model, where $\aleph_1 = \mad < \dom = \ser$.

\begin{question}
Is $\mad \leq \ser$?
\end{question}
We note that if $\ser < \mad$ is consistent, then proving it is likely very difficult. This is because $\ser < \mad$ implies $\dom < \mad$. The consistency of $\dom < \mad$ was an open question for a long time, solved by Shelah in \cite{Shelah}.
Shelah's technique will not work, however, for obtaining a model of $\ser < \mad$: his technique makes $\ser$ large for the same reasons it makes $\mad$ large. (Roughly, by an ``averaging of names'' argument, an ultrapower $\PP^\k/\mathcal U$ forces that if $\ser \geq \k$ then $\ser = \continuum$.)

\section{Forcing an (almost) arbitrary spectrum}

Every $\k \in \bspec$ is an example what Blass calls a ``$\mathbf{\Delta}^1_1$ characteristic'' in \cite{Blass}. Blass proves \cite[Theorems 8 and 9]{Blass} that there can be many cardinals between $\aleph_1$ and $\continuum$ that are not $\mathbf{\Delta}^1_1$ characteristics, and therefore are not in $\bspec$. For example, it is consistent to have the set of $\mathbf{\Delta}^1_1$ characteristics be equal to precisely $\set{\aleph_n}{n \text{ is a power of }17} \cup \{\aleph_\w,\aleph_{\w+1}\}$. Blass' method does not guarantee that any of these $\mathbf{\Delta}^1_1$ characteristics will be in $\bspec$, except of course for $\aleph_1$ and $\continuum$. But his results do show at least that $\bspec$ can contain large gaps, and that it is possible to surgically exclude specific cardinals from $\bspec$. (We should mention that the consistency of having $\aleph_2 \notin \bspec$ predates Blass' work, and is due to Miller \cite{Miller2}.)

In the other direction, extending an earlier result of his \cite[Theorem 4]{Miller}, Arnie Miller and the author showed in \cite[Theorem 3.12]{Brian&Miller} that:
\begin{theorem}\label{thm:BM}
For any $\k \geq \continuum$ with $\mathrm{cf}(\k) > \w$, there is a ccc forcing extension in which $\mathfrak{sp}(\text{\emph{\small closed}}) = [\aleph_1,\k]$.
\end{theorem}
\noindent In particular, for any set $C$ of cardinals, there is a ccc forcing extension in which $C \sub \Spec$. 
We now sketch a proof of this theorem, in order to introduce some of the ideas used in the proof of the main theorem below, but in a simpler context.  
This sketch can be skipped by the reader requiring no such introduction.

\begin{proof}[Proof sketch of Theorem~\ref{thm:BM}]
Given $X \sub 2^\w$, let $\TT_{\!X}$ be the poset whose conditions are pairs $(t,B)$, where
\begin{itemize}
\item[$\circ$] There is  some $k \in \w$ such that $t$ is a subtree of $2^{<k}$, and $t$ is pruned: i.e., if $\s \in t$ with $|\s| < k$, then $\s$ has a proper extension in $t$.
\item[$\circ$] $B$ is a finite subset of $2^\w \setminus X$, and $b \rest k$ is a branch of $t$ for every $b \in B$.
\end{itemize}
The ordering on $\TT_{\!X}$ is defined by: $(t',B')$ extends $(t,B)$ if and only if $B' \supseteq B$ and $t'$ is an end extension of $t$ (meaning that $t = t' \cap 2^{<k}$ for some $k$).

For any $X \sub 2^\w$, $\TT_{\!X}$ is $\s$-centered.
The poset $\TT_{\!X}$ generically adds an infinite pruned subtree $T$ of $2^{<\w}$, defined from a generic filter $G$ on $\TT_{\!X}$ as $T = \bigcup \set{t}{(t,B) \in G \text{ for some } B}$. (Equivalently, $T$ is the evaluation in $V[G]$ of the name
$\dot T = \set{\langle \s,q \rangle}{q = (t,B) \text{ for some } B \text{, and } \s \in t}.$)
In the extension, $\tr{T}$ is a closed subset of $2^\w$ disjoint from $X$.

Let $\TT_{\!X}^{\,\w}$ denote the finite support product of countably many copies of $\TT_{\!X}$.
For any $X \sub 2^\w$, $\TT_{\!X}^{\,\w}$ is $\s$-centered. 
The poset $\TT_{\!X}^{\,\w}$ generically adds countably many infinite pruned subtrees $T_0,T_1,T_2,\dots$ of $2^{<\w}$, and in the extension,
$$\textstyle (2^\w)^V \cap \bigcup_{n < \w}\tr{T_n} \,=\, (2^\w)^V \setminus X.$$
In other words, $\TT_{\!X}^{\,\w}$ generically adds an $F_\s$ subset of $2^\w$ that intersects the ground model reals in precisely the complement of $X$. (Note: this poset may be familiar: it is the one usually used for showing that $\mathsf{M}$artin's $\mathsf{A}$xiom implies every $<\!\continuum$-sized subset of $2^\w$ is a $Q$-set.)

Let $C$ denote the set of cardinals in $[\aleph_1,\k]$.
We now define a finite support iteration of length $\w_1$ as follows. 
At stage $0$, force with the poset $\Q_0 = \CC_\k$ of finite partial functions $\k \to 2$, in order to add a set of $\k$ mutually generic Cohen reals. 
Let $\set{c_\xi}{\xi < \k}$ be an enumeration of these Cohen reals in $V^{\Q_0}$, and
for each $\mu \in C$, let $\P_0^\mu = \set{\{c_\xi\}}{\xi < \mu}$. Note that $\P^\mu_0$ is a $\mu$-sized collection of disjoint subsets of $\R$: we think of $\P^\mu_0$ as a first approximation to a $\mu$-sized partition we are trying to build.
At a later stage $\a$ of the iteration, suppose we have already obtained, for each $\mu \in C$, a $\mu$-sized collection $\P^\mu_\a$ of disjoint subsets of $\R$.
In $V^{\Q_\a}$, define $X^\mu_\a = \bigcup \P^\mu_\a$ for each $\mu \in C$, and then obtain $V^{\Q_{\a+1}}$ from $V^{\Q_\a}$ by forcing with $\prod_{\mu \in C}\TT^{\,\w}_{\!X_\a^\mu}$.
This adds countably many generic trees $T^\mu_{\a,0},T^\mu_{\a,1},T^\mu_{\a,2},\dots$ for each $\mu \in C$, and in $V^{\Q_{\a+1}}$ we define $\P^\mu_{\a+1} = \P^\mu_\a \cup \set{\tr{T^\mu_{\a,n}}}{n \in \w}$.

At the end of the iteration, in $V^{\Q_{\w_1}}$, let $\P^\mu = \bigcup_{\a < \w_1}\P_\a^\mu$. This is a $\mu$-sized collection of disjoint $F_\s$ subsets of $\R$. Furthermore, if $x$ is a real in $V^{\Q_{\w_1}}$, then there is some $\a < \w_1$ with $x \in V^{\Q_\a}$. At that stage of the iteration, either $x \in \bigcup \P^\mu_\a$, or if not, then $x \in \bigcup_{n \in \w} \tr{T^\a_{\mu,n}}$ because $\textstyle (2^\w)^{V^{\PP_\a}} \cap \bigcup_{n < \w}\tr{T^\a_{\mu,n}} \,=\, (2^\w)^{V^{\PP_\a}} \setminus X_\a^\mu$. Either way, $x \in \bigcup \P^\mu_{\a+1} \sub \bigcup \P^\mu$. Thus $\P^\mu$ is a partition of $2^\w$ into $F_\s$ sets in $V^{\Q_{\w_1}}$, and this means $\mu \in \Spec$ by Corollary~\ref{cor:Spec}.
\end{proof}

The main theorem is proved with a modification of this poset, with two major changes. First, the set $C$ will not necessarily be an interval of cardinals, but will consist only of those cardinals we wish to add to $\Spec$. 
Second, instead of a true iteration, we use a streamlined modification.
This modified iteration has actual finite sequences (rather than names for them) for the working parts of the $\TT_{\!X_\a^\mu}$, which seems to be necessary for proving that certain permutations of the Cohen reals extend to automorphisms of the entire poset. (And these automorphisms are essential for excluding cardinals $\notin \!C$ from $\bspec$.)
The definition of this modified iteration is reminiscent of the template forcing notions in \cite{Brendle1}, but with a well ordered template, so that it is essentially an iteration. 

\begin{theorem}\label{thm:main}
Let $C$ be a set of uncountable cardinals such that 
\begin{itemize}
\item[$\circ$] $\min (C)$ is regular,

\vspace{.5mm}

\item[$\circ$] $|C| < \min (C)$,

\vspace{.5mm}

\item[$\circ$] $C$ has a maximum with $\mathrm{cf}(\max (C)) > \w$,

\vspace{.5mm}

\item[$\circ$] $C$ is closed under singular limits, and

\vspace{.5mm}

\item[$\circ$] if $\lambda$ is singular and $\lambda \in C$, then $\lambda^+ \in C$.
\end{itemize}
Assuming \gch holds up to $\max (C)$, there is a ccc forcing extension in which $\mathfrak{sp}(\text{\emph{\small closed}}) = C$, and furthermore, if $\min(C) < \lambda \notin C$, then $\lambda \notin \bspec$.
\end{theorem}

Before proving this theorem, let us deduce some relatively easy corollaries from it, including the results mentioned in the introduction.

\begin{corollary}\label{cor:0}
Let $C$ be a set of uncountable cardinals such that 
\begin{itemize}
\item[$\circ$] $\aleph_1 \in C$,

\vspace{.5mm}

\item[$\circ$] $C$ is at most countable,

\vspace{.5mm}

\item[$\circ$] $C$ has a maximum with $\mathrm{cf}(\max (C)) > \w$,

\vspace{.5mm}

\item[$\circ$] $C$ is closed under singular limits, and

\vspace{.5mm}

\item[$\circ$] if $\lambda$ is singular and $\lambda \in C$, then $\lambda^+ \in C$.
\end{itemize}
Assuming \gch holds up to $\max (C)$, there is a ccc forcing extension in which $\mathfrak{sp}(\text{\emph{\small Borel}}) = \mathfrak{sp}(\text{\emph{\small closed}}) = C$.
\end{corollary}
\begin{proof}
This follows immediately from the previous theorem.
\end{proof}

\begin{corollary}\label{cor:1}
Given any $A \sub \w \setminus \{0\}$, there is a forcing extension in which $\mathfrak{sp}(\text{\emph{\small Borel}}) = \set{\aleph_n}{n \in A} \cup \{\aleph_1,\aleph_\w,\aleph_{\w+1}\}$, and there is a forcing extension in which $\mathfrak{sp}(\text{\emph{\small closed}}) = \set{\aleph_n}{n \in A} \cup \{\aleph_\w,\aleph_{\w+1}\}$.
\end{corollary}
\begin{proof}
Fix some $A \sub \w \setminus \{0\}$.
First pass to a forcing extension in which \gch holds up to $\aleph_{\w+1}$.
Then, for the result about $\Spec$, apply Theorem~\ref{thm:main} with $C = \set{\aleph_n}{n \in A} \cup \{\aleph_\w,\aleph_{\w+1}\}$. Or, for the result about $\bspec$, apply Corollary~\ref{cor:0} with $C = \set{\aleph_n}{n \in A} \cup \{\aleph_1,\aleph_\w,\aleph_{\w+1}\}$.
\end{proof}

\begin{corollary}\label{cor:2}
Given any finite $A \sub \w \setminus \{0\}$, there is a forcing extension in which $\mathfrak{sp}(\text{\emph{\small Borel}}) = \set{\aleph_n}{n \in A} \cup \{\aleph_1\}$, and there is a forcing extension in which $\mathfrak{sp}(\text{\emph{\small closed}}) = \set{\aleph_n}{n \in A}$.
\end{corollary}
\begin{proof}
This is proved in exactly the same way as the previous corollary, but without $\aleph_\w$ and $\aleph_{\w+1}$ put into $C$.
\end{proof}

\begin{proof}[Proof of Theorem~\ref{thm:main}]
Let $C$ be a set of uncountable cardinals satisfying the hypotheses listed in the statement of the theorem, and assume \gch holds up to $\max (C)$.
Let $\k = \min (C)$ and let $\theta = \max (C)$.

For each $\mu \in C$, let $I_\mu = \mu \times \{\mu,\theta^{+}\}$. (These are merely indexing sets, and for all practical purposes one may think of each $I_\mu$ as a set of atoms, or urlements, in the set-theoretic universe. The relevant properties of these $I_\mu$'s are that they are pairwise disjoint, each $I_\mu$ has size $\mu$, and the $I_\mu$ do not ``interact'' in any accidental way with any other sets in the proof.) 

Let $I = \bigcup_{\mu \in C}I_\mu$, and let $\PP_0$ denote the poset of finite partial functions from $I \times \w$ to $2$. Note that this is equivalent to the standard poset $\CC_\theta$ for adding $\theta$ mutually generic Cohen reals. For each $i \in I$, let $c_i$ denote in $V^{\PP_0}$ the Cohen real added by $\PP_0$ in coordinate $i$. More formally, $c_i$ denotes in $V^{\PP_0}$ the evaluation of the name
$\dot c_i = \set{\<(n,j),p\>}{p(i,n) = j}.$

Next we define, recursively, a poset $\< \PP_\a,\leq_{\PP_\a} \>$ for each $\a \leq \k$ with $\a > 0$.
At every stage of the recursion, $\PP_\a$ is defined so that $\PP_\b$ is a sub-poset of $\PP_\a$. 
Conditions in $\PP_\a$ are finite partial functions on $(I \times \w) \cup (\a \times C \times \w)$, and the $\PP_\a$ are defined so that for any $\b < \a$, the restriction of any condition in $\PP_\a$ to $(I \times \w) \cup (\b \times C \times \w)$ is a condition in $\PP_\b$.

At limit stages, we take $\PP_\a = \bigcup_{\xi < \a}\PP_\xi$ and $\leq_{\PP_\a} = \bigcup_{\xi < \a}\leq_{\PP_\xi}$. In other words, $\< \PP_\a,\leq_{\PP_\a} \>$ is the direct limit of $\seq{\< \PP_\xi,\leq_{\PP_\xi} \>}{\xi < \a}$ for limit $\a$.

At successor stages, suppose $\PP_\xi$ is given for every $\xi \leq \b$, and $\PP_\xi \supseteq \PP_\z$ whenever $\z \leq \xi \leq \b$.
Let $\a = \b+1$. 
Conditions in $\PP_\a$ are finite partial functions $p$ on
$(I \times \w) \cup (\a \times C \times \w)$ 
such that:
\begin{itemize}
\item[$\circ$] if $(i,n) \in (I \times \w) \cap \mathrm{dom}(p)$, then $p(i,n) \in \{0,1\}$.

\vspace{1mm}

\item[$\circ$] if $(\g,\mu,n) \in \mathrm{dom}(p)$ for some $\g \leq \b$, $\mu \in C$, and $n \in \w$, then $p(\g,\mu,n) = (t,B)$, where

\vspace{.5mm}

\begin{itemize}
\item[\raisebox{.3mm}{\scriptsize$\circ$}] $t$ is a pruned subtree of $2^{<k}$ for some $k \in \w$ (in this context, ``pruned'' means that if $\s \in t$ and $|\s| < k$, then $\s$ has an extension in $t$).

\vspace{.5mm}

\item[\raisebox{.35mm}{\scriptsize$\circ$}] $B$ is a finite set of nice $\PP_\g$-names for memebers of $2^\w$ such that for every $\dot b \in B$,
\begin{align*}
\quad \quad \quad \quad \quad  p \rest ((I \times \w) \cup (\g \times C \times \w)) \ \forces_{\PP_\g} \quad & \dot b \rest j \in t \text{ for every } j < k,  \\
& \dot b \neq \dot c_i \text{ for every } i \in I_\mu, \text{ and}  \\
& \dot b \text{ is not a branch of } \dot T^\xi_{\mu,m} \\
& \quad \text{ for any } \xi < \g \text{ and } m \in \w
\end{align*}
where $\dot c_i$ is the $\PP_0$-name described above, but interpreted as a $\PP_\g$-name, and
where for every $\xi < \g$ and $m \in \w$, $\dot T^\xi_{\mu,m}$ is the $\PP_\g$-name for a subtree of $2^{<\w}$ defined as follows:
$$\quad \quad \quad \quad \ \  \dot T^\xi_{\mu,m} = \set{\langle \s,q \rangle}{q(\xi,\mu,m) = (t',B') \text{ for some } B' \text{, and } \s \in t'}.$$
\end{itemize}
\end{itemize}
%
The extension relation $\leq_{\PP_\a}$ on $\PP_\a$ is defined as follows: given $p,q \in \PP_\a$, we write $q \leq_{\PP_\a} p$ (meaning that $q$ extends $p$) if and only if 
\begin{itemize}
\item[$\circ$] $\mathrm{dom}(q) \supseteq \mathrm{dom}(p)$,
\item[$\circ$] $q \rest (I \times \w) \supseteq p \rest (I \times \w)$, and
\item[$\circ$] if $(\g,\mu,n) \in \mathrm{dom}(p)$ for some $\g \leq \b$, $\mu \in C$, and $n \in \w$, and if $p(\g,\mu,n) = (t,B)$ and $q(\g,\mu,n) = (t',B')$, then $B' \supseteq B$ and $t'$ is an end extension of $t$ (meaning that $t = t' \cap 2^{<k}$ for some $k$).
\end{itemize}
Naturally, we abbreviate ``$\leq_{\PP_\a}$'' with ``$\leq$'' in situations where this creates no ambiguity. 

This completes the recursive definition of the $\PP_\a$ and $\leq_{\PP_\a}$.
It is easy to see that $\leq_{\PP_\a}$ is a partial order on $\PP_\a$ for all $\a \leq \k$, and that $\PP_\b$ is a sub-poset of $\PP_\a$ whenever $\b \leq \a \leq \k$.

\begin{lemma}\label{lem:embedded}
$\PP_\b$ is a complete sub-poset of $\PP_\a$ for all $\b \leq \a \leq \k$.
\end{lemma}
\begin{proof}
Fix $\b \leq \a \leq \k$. It is clear that if $p,q \in \PP_\a$ and $p$ and $q$ are incompatible in $\PP_\a$, then $p \rest \b$ and $q \rest \b$ are incompatible in $\PP_\b$. Conversely, if two conditions are incompatible in $\PP_\b$, they remain incompatible in $\PP_\a$. 
So to prove the lemma, it suffices to show that every maximal antichain in $\PP_\b$ is also a maximal antichain in $\PP_\a$. Suppose $\A$ is a maximal antichain in $\PP_\b$, and let $p \in \PP_\a$. By the maximality of $\A$, there is some $q \in \A$ such that $q$ is compatible with $p \rest \b$ in $\PP_\b$. Let $r \in \PP_\b$ be a common extension of $q$ and $p \rest \b$. Define a function $s$ on $(I \times \w) \cup (\a \times \C \times \w)$ by setting $s=r$ on $(I \times \w) \cup (\b \times \C \times \w)$ and setting $s=p$ on $(I \times \w) \cup ((\a \setminus \b) \times \C \times \w)$. Then $s \in \PP_\a$, and $s$ is a common extension of $p$ and $q$ in $\PP_\a$. As $p$ was arbitrary, this shows $\A$ is a maximal antichain in $\PP_\a$.
\end{proof}

\begin{lemma}\label{lem:ccc}
$\PP_\k$ has the ccc.
\end{lemma}
\begin{proof}
Let $\A_0$ be an uncountable subset of $\PP_\k$. By the $\Delta$-system lemma, there is an uncountable $\A_1 \sub \A_0$ and a finite $R \sub (I \times \w) \cup (\k \times C \times \w)$ such that for any $p,q \in \A_1$, $\mathrm{dom}(p) \cap \mathrm{dom}(p) = R$. 
If $i \in R \cap (I \times \w)$, then $p(i) \in \{0,1\}$ for every $p \in \A_1$. As there are only finitely many functions $R \cap (I \times \w) \to 2$, this implies there is some uncountable $\A_2 \sub \A_1$ such that $p \rest (R \cap (I \times \w)) = q \rest (R \cap (I \times \w))$ for any $p,q \in \A_2$.
For every $p \in \A_2$, if $j \in R \setminus (I \times \w)$, then $p(j) = (t,B)$ for some finite subtree $t$ of $2^{<\w}$. As there are only countably many finite subtrees of $2^{<\w}$, there is some uncountable $\A_3 \sub \A_2$ such that for each $j \in R \setminus (I \times \w)$, there is some fixed $t_j$ such that for any $p \in \A_3$, $p(j) = (t_j,B^p_j)$ for some $B^p_j$.
But any two members of $\A_3$ are compatible: if $p,q \in \A_3$ then
$$r(j)=
\begin{cases}
p(j) & \text{ if } j \in \mathrm{dom}(p) \setminus R \\
q(j) & \text{ if } j \in \mathrm{dom}(q) \setminus R \\
(t_j,B^p_j \cup B^q_j) & \text{ if } j \in \mathrm{dom}(p) \cap \mathrm{dom}(q) = R
\end{cases}
$$
is a common extension of $p$ and $q$. Thus $\PP_\k$ has no uncountable antichains, and is ccc. (In fact we have shown a bit more: that $\PP_\k$ has property $K$.)
\end{proof}

If $\b \leq \a$, then because $\PP_\b \sub \PP_\a$, we may (and do) consider every $\PP_\b$-name to be a $\PP_\a$-name as well.
Let us also set the convention that in $V^{\PP_\a}$, the evaluation of a $\PP_\a$-name is indicated by removing its dot. So, for example, $c_i$ denotes in $V^{\PP_\a}$ (for any $\a \leq \k$) the Cohen real added by $\PP_0$ in coordinate $i$, and $T^\xi_{\mu,n}$ denotes in $V^{\PP_\a}$, for any $\a \geq \xi$, the evaluation of the name $\dot T^\xi_{\mu,n}$.

\begin{lemma}\label{lem:Spec}
$C \sub \mathfrak{sp}(\text{\emph{\small closed}})$ in $V^{\PP_\k}$.
\end{lemma}
\begin{proof}
Let $\mu \in C$.
We claim that in $V^{\PP_\k}$,
$$\P = \set{\{c_i\}}{\vphantom{2^i}i \in I_\mu} \cup \textstyle \set{\bigcup_{n \in \w}\tr{T^\a_{\mu,n}}}{\a < \k}$$
is a partition of $2^\w$. Observe that every member of $\P$ is an $F_\s$ subset of $2^\w$, and $|\P| = \mu$. (It is clear that each $T^\a_{\mu,n}$ is a subtree of $2^{<\w}$, which means the sets of the form $\bigcup_{n \in \w} \tr{T^\a_{\mu,n}}$ are all $F_\s$.) Thus if $\P$ is a partition, then $\mu \in \Spec$ by Theorem~\ref{thm:Spec}.

The members of $\P$ come in two types, so in order to show they are pairwise disjoint, we have three things to prove.

First, if $i,i' \in I_\mu$ and $i \neq i'$, then clearly $c_i \neq c_{i'}$, i.e., $\{c_i\} \cap \{c_{i'}\} = \0$.

Second, fix $i \in I_\mu$, $\a < \k$, and $n \in \w$. We claim $\{c_i\} \cap \tr{T^\a_{\mu,n}} = \0$, or equivalently, $c_i \notin \tr{T^\a_{\mu,n}}$.
To see this, fix some $p \in \PP_\k$: we will find an extension $r$ of $p$ forcing $c_i \notin \tr{T^\a_{\mu,n}}$. 
Extending $p$ to $p \cup \{\< (\a,\mu,n),(\0,\0) \>\}$, if necessary, we may (and do) assume $(\a,\mu,n) \in \mathrm{dom}(p)$. Let $(t,B) = p(\a,\mu,n)$, with $t$ a pruned subtree of $2^{<k}$. 
From the definition of $\PP_\k$, we know $p \rest \a \forces_{\PP_\a} \dot b \neq \dot c_i$ for every $\dot b \in B$.
Using this, and the fact that $B$ is finite, there is an extension $q$ of $p$ in $\PP_\a$, and some finite sequence $\s \in 2^\ell$ for some $\ell > k$, such that $q \forces_{\PP_\a}$ ``$\dot c_i \rest \ell = \s$ but $\dot b \rest \ell \neq \s$ for all $\dot b \in B$." 
Let $t'$ be the largest subtree of $2^{<\ell+1}$ that is an end extension of $t$ and that does not contain $\s$.
(In other words, $\t \in t'$ if and only if $\t \in 2^{<\ell+1} \setminus \{\s\}$ and $\t \rest k \in t$. This is a subtree of $2^{<\ell+1}$, and it is an end extension of $t$ because $\ell > k$ and $\s \in 2^\ell$.)
Define a function $r$ on $(I \times \w) \cup (\k \times \C \times \w)$ by setting 
$$r(j)=
\begin{cases}
q(j) & \text{ if } j \in (I \times \w) \cup (\a \times \C \times \w) \\
p(j) & \text{ if } j \in (I \times \w) \cup ((\k \setminus \a) \times \C \times \w) \text{ but } j \neq (\a,\mu,n)),  \\
(t',B) & \text{ if } j = (\a,\mu,n).
\end{cases}
$$
Then $r \in \PP_\k$, and $r \forces \s \notin \dot T^\a_{\mu,n}$ and $\dot c_i \rest \ell = \s$, which means $r \forces \dot c_i \notin \tr{\dot T^\a_{\mu,n}}$. As $p$ was arbitrary, this shows that $c_i \notin \tr{T^\a_{\mu,n}}$ in $V^{\PP_\k}$.

Third, fix $\b < \a < \k$, and $m,n \in \w$. We claim that $\tr{T^\b_{\mu,m}} \cap \tr{T^\a_{\mu,n}} = \0$.
To see this, fix $p \in \PP_\k$: we will find an extension of $p$ forcing $\tr{T^\b_{\mu,m}} \cap \tr{T^\a_{\mu,n}} = \0$. 
Extending $p$ to $p \cup \{\< (\b,\mu,m),(\0,\0) \>,\< (\a,\mu,n),(\0,\0) \>\}$, if necessary, we may (and do) assume $(\b,\mu,m),(\a,\mu,n) \in \mathrm{dom}(p)$. 
Let $(t_\b,B_\b) = p(\b,\mu,m)$, where $t_\b$ is a pruned subtree of $2^{<k_\b}$ for some $k_\b \in \w$, and let $(t_\a,B_\a) = p(\a,\mu,n)$, where $t_\a$ is a pruned subtree of $2^{<k_\a}$ for some $k_\a \in \w$. 
From the definition of $\PP_\k$, we know that
$p \rest ((I \times \w) \cup (\a \times C \times \w)) \forces_{\PP_\a} \dot b \neq \dot a$ for every $\dot b \in B_\b$ and $\dot a \in B_\a$.
Using this, and the fact that $B_\b$ and $B_\a$ are both finite, there is an extension $q$ of $p$ in $\PP_\a$, and some $\ell > k_\b,k_\a$, such that $q$ ``decides'' all the $\dot b$ and $\dot a$ up to $\ell$, and in such a way that witnesses $\dot b \neq \dot a$ for all $\dot b \in B_\b$ and $\dot a \in B_\a$.
More precisely: for each $\dot b \in B_\b$, there is a particular branch $c(\dot b)$ of $2^{<\ell+1}$ such that $q \forces_{\PP_\a}$ ``$\dot b \rest \ell = c(\dot b)$''; and similarly, for each $\dot a \in B_\a$, there is a particular branch $c(\dot a)$ of $2^{<\ell+1}$ such that $q \forces_{\PP_\a}$ ``$\dot a \rest \ell = c(\dot a)$''; and furthermore, $c(\dot b) \rest \ell \neq c(\dot a) \rest \ell$ for all $\dot b \in B_\b$ and $\dot a \in B_\a$.
Let $t'_\b$ be the largest end extension of $t_\b$ that is a pruned subtree of $2^{<\ell}$, and let $t''_\b$ be the end extension of $t'_\b$ to a pruned subtree of $2^{<\ell+1}$ containing on level $\ell$ the nodes:
\begin{align*}
\s \cat 0 \quad & \text{ if } \s \in t'_\b, \s = c(\dot a) \text{ for some } \dot a \in B_\a, \text{ and } c(\dot a)(\ell) = 1  \\
\s \cat 1 \quad & \text{ if } \s \in t'_\b, \s = c(\dot a) \text{ for some } \dot a \in B_\a, \text{ and } c(\dot a)(\ell) = 0  \\
\s \cat 0 \quad & \text{ if } \s \in t'_\b \text{ and } \s \neq c(\dot a) \text{ for any } \dot a \in B_\a.
\end{align*}
Similarly, let $t'_\a$ be the largest end extension of $t_\a$ that is a pruned subtree of $2^{<\ell}$, and let $t''_\a$ be the end extension of $t'_\a$ to a pruned subtree of $2^{<\ell+1}$ containing on level $\ell$ the nodes:
\begin{align*}
\s \cat 0 \quad & \text{ if } \s \in t'_\a, \s = c(\dot b) \text{ for some } \dot b \in B_\b, \text{ and } c(\dot b)(\ell) = 1  \\
\s \cat 1 \quad & \text{ if } \s \in t'_\a, \s = c(\dot b) \text{ for some } \dot b \in B_\b, \text{ and } c(\dot b)(\ell) = 0  \\
\s \cat 1 \quad & \text{ if } \s \in t'_\a \text{ and } \s \neq c(\dot b) \text{ for any } \dot b \in B_\b.
\end{align*}
Observe that $t''_\b$ and $t''_\a$ have no common nodes on level $\ell$. 
Define a function $r$ on $(I \times \w) \cup (\k \times \C \times \w)$ by setting 
$$r(j)=
\begin{cases}
q(j) & \text{ if } j \in (I \times \w) \cup (\a \times \C \times \w) \text{ but } j \neq (\b,\mu,m),  \\
p(j) & \text{ if } j \in (I \times \w) \cup ((\k \setminus \a) \times \C \times \w)) \text{ but } j \neq (\a,\mu,n),  \\
(t''_\b,B_\b) & \text{ if } j = (\b,\mu,m)  \\
(t''_\a,B_\a) & \text{ if } j = (\a,\mu,n).
\end{cases}
$$
Then $r \in \PP_\k$, and $r \forces t''_\b$ is a subtree of $\dot T^\b_{\mu,m}$ and $r \forces t''_\a$ is a subtree of $\dot T^\a_{\mu,n}$. Because $t''_\b$ and $t''_\a$ share no nodes on level $\ell$, $r \forces \tr{\dot T^\b_{\mu,m}} \cap \tr{\dot T^\a_{\mu,n}} = \0$. As $p$ was arbitrary, this shows that $\tr{T^\b_{\mu,m}} \cap \tr{T^\a_{\mu,n}} = \0$ in $V^{\PP_\k}$.

Thus $\P$ is a pairwise disjoint collection of $F_\s$ subsets of $2^\w$ in $V^{\PP_\k}$. To finish the proof, we must show also that $\bigcup \P = 2^\w$.

Fix $x \in 2^\w$ in $V^{\PP_\k}$. Because $\PP_\k$ has the ccc, and each $\PP_\a$ is a complete sub-poset of $\PP_\k$, there is some $\a < \k$ such that $x \in V^{\PP_\a}$.
If $x = c_i$ for some $i \in I_\mu$, or if $x \in \tr{T^\xi_{\mu,n}}$ for some $\xi < \a$, then we're done. 
If not, let $\dot x$ be a nice $\PP_\a$-name for $x$, and fix some $p \in \PP_\k$ such that $p \forces$ ``$\dot x \neq \dot c_i$ for all $i \in I_\mu$, and $\dot x \notin \tr{\dot T^\xi_{\mu,n}}$ for all $\xi < \a$ and $n \in \w$." 
Because $p$ has finite support, there is some $N \in \w$ such that $(\a,\mu,N) \notin \mathrm{dom}(p)$. But then $q = p \cup \{ \< (\a,\mu,N),(\0,\{\dot x\}) \> \}$ is in $\PP_\a$, and $q \forces \dot x \in \tr{\dot T^\a_{\mu,N}}$. 
Thus for every condition $p$ forcing $x \notin \bigcup_{i \in I_\mu}\{c_i\} \cup \bigcup_{\xi < \a}\bigcup_{n \in \w}\tr{T^\xi_{\mu,n}}$, there is an extension $q$ of $p$ forcing $x \in \bigcup_{n \in \w}\tr{T^\a_{\mu,n}}$. Hence $x \in \bigcup_{i \in I_\mu}\{c_i\} \cup \bigcup_{\xi \leq \a}\bigcup_{n \in \w}\tr{T^\xi_{\mu,n}} \sub \bigcup \P$. As $x$ was arbitrary, $\bigcup \P = 2^\w$ as claimed.
\end{proof}

\begin{lemma}
$\continuum = \theta = \max(C)$ in $V^{\PP_\k}$.
\end{lemma}
\begin{proof}
A straightforward transfinite induction on $\a$ shows that $|\PP_\a| = \theta$ for all $\a \leq \k$. 
(For the base case, clearly $|\PP_0| = \theta$, and the limit case is also clear. For the successor case, suppose $|\PP_\a| = \theta$. Because $\PP_\a$ has the ccc, there are $\theta^{\aleph_0} = \theta$ nice $\PP_\a$-names for reals, and it follows that $|\PP_{\a+1}| = \theta$.)
In particular, $|\PP_\k| = \theta$, and because $\PP_\k$ has the ccc, it follows that there are $\theta^{\aleph_0} = \theta$ nice $\PP_\k$-names for reals. Hence $\continuum \leq \theta$ in $V^{\PP_\k}$.

On the other hand, it is obvious that $\continuum \geq \theta$ in $V^{\PP_\k}$, because of the Cohen reals added by $\PP_0$.
\end{proof}

\begin{lemma}
$\mathrm{cov}(\mathcal M) = \dom = \ser = \k = \min(C)$ in $V^{\PP_\k}$.
\end{lemma}
\begin{proof}
By Lemma~\ref{lem:Spec} and Corollary~\ref{cor:serd}, $\k \geq \ser \geq \dom \geq \mathrm{cov}(\mathcal M)$ in $V^{\PP_\k}$. So to prove the lemma, it suffices to show that $2^\w$ cannot be covered with $<\!\k$ meager sets in $V^{\PP_\k}$.

Suppose $\lambda < \k$ and $\set{C_\xi}{\xi < \lambda}$ is a collection of closed nowhere dense subsets of $2^\w$ in $V^{\PP_\k}$. For each $C_\xi$, let $A_\xi$ be a Borel code that evaluates to $C_\xi$. 
Because $\PP_\k$ has the ccc and each $\PP_\a$ is a complete sub-poset of $\PP_\k$, there is for each $\xi$ some $\a < \k$ such that $A_\xi \in V^{\PP_\a}$.
Because $\k$ is regular and $\lambda < \k$, there is some particular $\a < \k$ such that $A_\xi \in V^{\PP_\a}$ for all $\xi < \lambda$.
So to prove the lemma, it suffices to show that for each $\a < \k$ there is some $c \in V^{\PP_\k}$ not contained in any closed nowhere dense subset of $2^\w$ coded in $V^{\PP_\a}$ (i.e., we want $c$ to be Cohen-generic over $V^{\PP_\a}$).

Fix $\a < \k$, and define $c: \w \to 2$ by setting
$$c(n) = 0 \quad \text{if and only if} \quad \langle 0 \rangle \in T^\a_{\k,n}.$$
We claim that $c$ is Cohen-generic over $V^{\PP_\a}$. 
To see this, let $U \sub 2^\w$ be a dense open set whose Borel code is in $V^{\PP_\a}$. In particular, $\tilde U = \set{\t \in 2^{<\w}}{\tr{\t} \sub U} \in V^{\PP_\a}$ (where, abusing notation slightly, $\tr{\t}$ denotes all those $x$ in $2^\w$ with $x \rest \mathrm{dom}(\t) = \t$). 
Now fix $p \in \PP_\k$. We will find an extension $r$ of $p$ forcing that $c \in U$.
Extending $p$ if necessary, we may (and do) assume that there is some $N \in \N$ such that $(\a,\mu,n) \in \mathrm{dom}(p)$ for all $n < N$, but $(\a,\mu,n) \notin \mathrm{dom}(p)$ for all $n \geq N$. For all $n < N$, let $p(\a,\mu,n) = (t_n,B_n)$. Define $\s \in 2^N$ by setting
$$\s(n) = 0 \quad \text{if and only if} \quad \langle 0 \rangle \in t_n$$
for all $n < N$. Because $U$ is dense in $2^\w$, there is some $\t \in \tilde U$ such that $\t \rest N = \s$. 
Because $\tilde U \in V^{\PP_\a}$, there is some $q \in \PP_\a$ with $q \leq_{\PP_\a} p \rest \a$ forcing that $\t \in \tilde U$.
Define $r \in \PP_\k$ by
$$r(j) = 
\begin{cases}
r(j) & \text{if } j \in (I \times \w) \cup (\a \times C \times \w),  \\
(\<\t(n)\>,\0) & \text{if } j = (\a,\mu,n) \text{ for some } n \in \mathrm{dom}(\t) \setminus \mathrm{dom}(\s), \\
p(j) & \text{otherwise.}
\end{cases}$$
Then $r \in \PP_\k$, and $r$ forces $c \in U$. As $U$ was an arbitrary open dense subset of $2^\w$ coded in $V^{\PP_\a}$, this shows $c$ is Cohen-generic over $V^{\PP_\a}$.
\end{proof}

From the last two lemmas, it follows that $\Spec \sub [\k,\theta]$ in $V^{\PP_\k}$, and if $\lambda > \theta$ then $\lambda \notin \bspec$. To finish proving the theorem, it remains only to show that if $\lambda \in [\k,\theta]$ and $\lambda \notin C$, then $\lambda \notin \bspec$. This is where the isomorphism-of-names argument comes in. This argument requires a rich supply of automorphisms of $\PP_\k$, which we describe next.

If $\phi: I \to I$ is a permutation, then for every set $x$ we define $\bar \phi(x)$ to be the set obtained from $x$ by replacing every $i \in I$ in the transitive closure of $x$ with $\phi(i)$. This is well defined, because if $i,j \in I$ then $i$ is not in the transitive closure of $j$. (One may think of $I$ as a set of urelements, $\phi$ as a permutation of them, and $\bar \phi$ as the automorphism of the universe induced by $\phi$.) Alternatively, $\bar \phi$ is described via well-founded recursion by the relation
$$\bar \phi(x) = \set{\bar \phi(y)}{y \in x}.$$
In particular, if $p \in \PP_\k$, then 
$$\bar \phi(p)(j) = 
\begin{cases}
p(i) & \text{if } i \in (I \times \w) \cap \mathrm{dom}(p) \text{ and } \phi(i) = j,  \\
\big(t,\{ \bar \phi(\dot b) :\, \dot b \in B \}\big) & \text{if } j \in \mathrm{dom}(p) \setminus (I \times \w) \text{ and } p(j) = (t,B),
\end{cases}$$
and this expression can be taken as a recursive definition of $\bar \phi$ on $\PP_\k$, and the natural extension of $\bar \phi$ to the class of $\PP_\k$-names.

\begin{lemma}\label{lem:permutations}
Let $\phi$ be a permutation of $I$ such that $\phi \rest I_\mu$ is a permutation of $I_\mu$ for all $\mu \in C$. 
Then $\bar \phi \rest \PP_\a$ is an automorphism of $\PP_\a$ for all $\a \leq \k$.
\end{lemma}
\begin{proof}
First, note that for all $p,q \in \PP_\a$, we have
$$\<p,q\> \in \ \leq_{\PP_\a} \quad \text{if and only if } \quad \bar \phi(\<p,q\>) = \<\bar \phi(p),\bar \phi(q)\> \in \bar \phi(\leq_{\PP_\a}).$$
Together with the fact that $\bar \phi$ is invertible (with inverse $\bar \phi^{-1} = \overline{\phi^{-1}}$), this shows that the image of $\PP_\a$ under $\bar \phi$ is a poset, and is naturally isomorphic to $\PP_\a$ as witnessed by $\bar \phi$.
But it is not immediately clear that the image of $\PP_\a$ under $\bar \phi$ is equal to $\PP_\a$, which is of course required for our claim that $\bar \phi \rest \PP_\a$ is an automorphism of $\PP_\a$. (In fact, one may show that this would not be true for $\a \geq 1$ if $\phi$ were not required to fix all the $I_\mu$.)
This is proved for $\bar \phi$ and its inverse, simultaneously, by induction on $\a$. 

The base case $\a=0$ is clear. And for limit $\a$, if $\bar \phi \rest \PP_\xi$ is an automorphism of $\PP_\xi$, for all $\xi < \a$, then $\bar \phi \rest \PP_\a$ is an automorphism of $\PP_\a$ because $\< \PP_\a,\leq_{\PP_\a} \>$ is the direct limit of $\seq{\< \PP_\xi,\leq_{\PP_\xi} \>}{\xi < \a}$. The same applies to $\bar \phi^{-1}$.

For the successor case, fix some $\b < \k$ and suppose $\bar \phi \rest \PP_\xi$ is an automorphism of $\PP_\xi$ for all $\xi \leq \b$ (the inductive hypothesis). Let $\a = \b+1$, and fix $p \in \PP_\a$. We claim that $\bar \phi(p) \in \PP_\a$. Recall that
$$\bar \phi(p)(j) = 
\begin{cases}
p(i) & \text{if } i \in (I \times \w) \cap \mathrm{dom}(p) \text{ and } \phi(i) = j,  \\
\big(t,\{ \bar \phi(\dot b) :\, \dot b \in B \}\big) & \text{if } j \in \mathrm{dom}(p) \setminus (I \times \w) \text{ and } p(j) = (t,B).
\end{cases}$$

It is clear that $\bar \phi(p)$ is a finite partial function on $(I \times \w) \cup (\a \times C \times \w)$, and that if $i \in (I \times \w) \cap \mathrm{dom}(\bar \phi(p))$, then $\bar \phi(p)(i) = p(\phi^{-1}(i)) \in \{0,1\}$.
It remains to check that the last bullet point in the definition of the conditions in $\PP_\a$ is satisfied by $\bar \phi(p)$.

Suppose $(\g,\mu,n) \in \mathrm{dom}(\bar \phi(p))$ for some $\g \leq \b$, $\mu \in C$, and $n \in \w$.
This means $(\g,\mu,n) \in \mathrm{dom}(p)$ as well: let us denote $p(\g,\mu,n) = (t,B)$. 
Then 
$$\bar \phi(p)(\g,\mu,n) = \big( t, \{ \bar \phi(\dot b) :\, \dot b \in B \} \big).$$
Because $p \in \PP_\a$, $t$ is a pruned subtree of $2^{<k}$ for some $k \in \w$, and (by the inductive hypothesis, that $\bar \phi \rest \PP_\g$ is an automorphism of $\PP_\g$), $\{ \bar \phi(\dot b) :\, \dot b \in B \}$ is a finite set of nice $\PP_\g$-names for members of $2^\w$. Furthermore, by applying the automorphism $\bar\phi$ of $\PP_\g$ to the displayed statement in the last bullet point in the definition of the conditions in $\PP_\a$, we get that for every $\dot b \in B$,
\begin{align*}
\bar \phi(p) \restriction ((I \times \w) \cup (\g \times C \times \w)) \ \forces_{\PP_\g} \quad  & \bar \phi(\dot b) \rest j \in t \text{ for every } j < k,  \\
& \bar \phi(\dot b) \neq \bar \phi(\dot c_i) \text{ for every } i \in I_\mu, \text{ and}  \\
& \bar \phi(\dot b) \text{ is not a branch of } \bar \phi(\dot T^\xi_{\mu,m}) \\
& \quad \text{ for any } \xi < \g \text{ and } m \in \w.
\end{align*}

Considering the second of these three statements, note that $\bar \phi(\dot c_i) = \dot c_{\phi(i)}$ for every $i \in I$. Because $\phi \rest I_\mu$ is a permutation of $I_\mu$, this means that
the assertion ``$\bar \phi(\dot b) \neq \bar \phi(\dot c_i) \text{ for every } i \in I_\mu$'' is equivalent to the assertion ``$\bar \phi(\dot b) \neq \dot c_i \text{ for every } i \in I_\mu$.'' Thus
\begin{align*}
\bar \phi(p) \restriction ((I \times \w) \cup (\g \times C \times \w)) \ \forces_{\PP_\g} \quad & \bar \phi(\dot b) \neq \dot c_i \text{ for every } i \in I_\mu.
\end{align*}
Considering the third of these three statements, observe that
\begin{align*}
\bar \phi(\dot T^\xi_{\mu,n}) &= \set{\bar \phi( \langle \s, p \rangle )}{p(\xi,\mu,m) = (t,B) \text{ for some } B \text{, and } \s \in t}  \\
& = \set{ \langle \s, \bar \phi(p) \rangle }{p(\xi,\mu,m) = (t,B) \text{ for some } B \text{, and } \s \in t}  \\
& = \set{ \langle \s, \bar \phi(p) \rangle }{\bar \phi(p)(\xi,\mu,m) = (t,\bar \phi(B)) \text{ for some } B \text{, and } \s \in t} \\
& = \set{ \langle \s, \bar \phi(p) \rangle }{\bar \phi(p)(\xi,\mu,m) = (t,B') \text{ for some } B' \text{, and } \s \in t} \\
& = \set{ \langle \s, q \rangle }{q(\xi,\mu,m) = (t,B') \text{ for some } B' \text{, and } \s \in t} \\
& = \dot T^\xi_{\mu,n}.
\end{align*}
The third equality is true because for any condition $p$, $p(\xi,\mu,m) = (t,B)$ if and only if $\bar \phi(p)(\xi,\mu,m) = (t,\bar \phi(B))$.
The fourth equality uses the inductive hypothesis, that $\bar \phi \rest \PP_\g$ is an automorphism: $p(\xi,\mu,m) = (t,\bar \phi(B))$ for some $B$ if and only if $p(\xi,\mu,m) = (t,B')$ for some $B'$, specifically for $B' = \bar \phi^{-1}(B)$. The fifth equality also uses the fact that $\bar \phi \rest \PP_\g$ is an automorphism.
Thus
\begin{align*}
\bar \phi(p) \restriction ((I \times \w) \cup (\g \times C \times \w)) \ \forces_{\PP_\g} \quad 
& \bar \phi(\dot b) \text{ is not a branch of } \dot T^\xi_{\mu,m} \\
& \quad \text{ for any } \xi < \g \text{ and } m \in \w.
\end{align*}

Putting these together, we obtain
\begin{align*}
\bar \phi(p) \restriction ((I \times \w) \cup (\g \times C \times \w)) \ \forces_{\PP_\g} \quad  & \bar \phi(\dot b) \rest j \in t \text{ for every } j < k,  \\
& \bar \phi(\dot b) \neq c_i \text{ for every } i \in I_\mu, \text{ and}  \\
& \bar \phi(\dot b) \text{ is not a branch of } \dot T^\xi_{\mu,m} \\
& \quad \text{ for any } \xi < \b \text{ and } m \in \w.
\end{align*}
In other words, $\bar \phi(p)$ satisfies the final bullet point in the definition of the $\PP_\a$ conditions. Hence $\bar \phi(p) \in \PP_\a$, as claimed.

This shows that (the restriction of) $\bar \phi$ is an injective morphism from $\PP_\a$ to $\PP_\a$. To get surjectivity, simply note that the same argument applies to $\bar \phi^{-1} = \overline{\phi^{-1}}$ as well. 
\end{proof}

From now on, we work in the ground model.
Let $\lambda$ be a cardinal such that $\lambda \in [\k,\theta]$ and $\lambda \notin C$.
Suppose $\big\{ \dot B_\a :\, \a < \lambda \big\}$ is a set of nice $\PP_\k$-names for Borel codes for subsets of $\R$. (Any standard method of constructing Borel codes can be used for the proof, provided only that the codes are hereditarily countable sets. But for concreteness, let us take a Borel code to be a subset of $\w$.) We let $B_\a$ denote the evaluation of the name $\dot B_\a$ in $V^{\PP_\k}$, and we let $\tilde B_\a \sub \R$ denote the interpretation of $B_\a$. 

We aim to show that $\big\{ \tilde B_\a :\, \a < \lambda \big\}$ is not a partition of $\R$ in $V^{\PP_\k}$.
To this end, suppose $q \in \PP_\k$ and $q$ forces each of the $\tilde B_\a$ is nonempty, and $\tilde B_\a \cap \tilde B_\b = \0$ whenever $\a \neq \b$.
We will show that $q$ also forces $\bigcup_{\a < \lambda} \tilde B_\a \neq \R$.

Fix a cardinal $\nu \leq \lambda$ such that $\a^{\aleph_0} < \nu$ for all cardinals $\a < \nu$, and such that $C$ contains no cardinals in the interval $[\nu,\lambda]$. If $\lambda$ is neither singular nor the successor of a singular cardinal, then we may simply take $\lambda = \nu$. Otherwise, using the last two bullet points in our description of $C$, there is an infinite interval of cardinals below $\lambda$ and disjoint from $C$, and we may take $\nu$ to be any successor-of-a-successor cardinal in this interval. In either case, $\a^{\aleph_0} < \nu$ for all cardinals $\a < \nu$ by the \gch.

Given $\a \leq \k$ and a condition $p \in \PP_\a$, for each $\mu \in C$ define 
$$\mathrm{hdom}_\mu(p) = \textstyle \text{TC}(p) \cap I_\mu,$$
where $\text{TC}(p)$ denotes the transitive closure of $p$.
We think of $\mathrm{hdom}_\mu(p)$ as the ``hereditary domain'' of $p$ on $I_\mu$: all those $i \in I_\mu$ that are used at any stage in building the condition $p$. 
A straightforward transfinite induction on $\a$ shows that $\card{\mathrm{hdom}_\mu(p)} \leq \aleph_0$ for every $p \in \PP_\a$ and every $\mu \in C$. 
(The base case and the limit case are clear. For the successor case, use the fact that $\PP_\a$ has the ccc.)
Similarly, for each $\PP_\k$-name $\dot x$ and each $\mu \in C$, let 
$\mathrm{hdom}_\mu(\dot x) = \textstyle \text{TC}(\dot x) \cap I_\mu$.
As before, it is not difficult to see that if $\dot x$ is a nice $\PP_\k$-name for a Borel code (or for any hereditarily countable set), then $\mathrm{hdom}_\mu(\dot x)$ is countable for every $\mu \in C$.

For each $\a < \lambda$ and each $\mu \in C$, let 
$D^\a_\mu = \mathrm{hdom}_\mu(\dot B_\a).$
Note that $D^\a_\mu$ is countable for each $\a < \lambda$ and $\mu \in C$.
Expanding some of the $D^\a_\mu$ if necessary, we may (and do) assume each $D^\a_\mu$ is countably infinite and includes $\mathrm{hdom}_\mu(q)$.
For each $\a < \lambda$, let $D^\a = \bigcup_{\mu \in C}D^\a_\mu$. 
We note that $|D^\a| = \sum_{\mu \in C}|D^\a_\mu| = |C| \cdot \aleph_0$ for all $\a < \lambda$, and $|C| \cdot \aleph_0 < \min (C) = \k$. (Note: The inequality $|C| \cdot \aleph_0 < \min(C)$ uses the second bullet point in our description of $C$.)

By our choice of $\nu$, together with the aforementioned fact that $|D^\a| < \k$ for all $\a < \lambda$, $\set{D^\a}{\a < \nu}$ meets the conditions of the generalized $\Delta$-system lemma \cite[Lemma III.6.15]{Kunen}.
Thus there is some $\A_0 \sub \nu$ with $\card{\A_0} = \nu$ such that $\set{D^\a}{\a \in \A_0}$ is a $\Delta$-system with root $R$. 

Let
$\A_1 = \set{\a \in \A_0}{ D^\a_\mu \setminus R = \0 \text{ for all } \mu \in C \text{ with } \mu < \lambda}$.
For all $\mu \in C$, $\set{D^\a_\mu \setminus R}{\a \in \A_0}$ is a $\nu$-size collection of pairwise disjoint subsets of $I_\mu$. 
If also $\mu < \lambda$, then $\mu < \nu$ and it follows that $D^\a_\mu \setminus R = \0$ for all but (at most) $\mu$ members of $\A_0$. 
Furthermore, $|\set{\mu \in C}{\mu < \lambda}| \leq |C| < \k < \nu$. It follows that $\card{\A_1} = \nu$. 

For each $\a < \lambda$ and $\mu \in C$, fix a bijection $\phi^\a_\mu: D^\a_\mu \to \w$, in such a way that $\phi^\a_\mu \rest R = \phi^\b_\mu \rest R$ for all $\a,\b < \lambda$. 
For each $\a,\b < \lambda$ and $\mu \in C$, let $\phi^{\a,\b}_\mu$ be the involution of $I_\mu$ given by
$$\phi^{\a,\b}_\mu(i) = 
\begin{cases}
(\phi^\b_\mu)^{-1} \circ \phi^\a_\mu(i) \quad & \text{if } i \in D^\a_\mu \\ 
(\phi^\a_\mu)^{-1} \circ \phi^\b_\mu(i) & \text{if } i \in D^\b_\mu, \\
i &\text{otherwise.}
\end{cases}$$
This function is well defined, because $D^\a_\mu \cap D^\b_\mu = R$, and $(\phi^\b_\mu)^{-1} \circ \phi^\a_\mu(i) = i = (\phi^\a_\mu)^{-1} \circ \phi^\b_\mu(i)$ whenever $i \in R$.
This gives us a collection $\set{\phi^{\a,\b}_\mu}{\a,\b < \lambda}$ of involutions of $I_\mu$ such that for any $\a,\b < \lambda$,
\begin{itemize}
\item[$\circ$] $\phi^{\a,\b}_\mu$ maps $D^\a_\mu$ onto $D^\b_\mu$ and $D^\b_\mu$ onto $D^\a_\mu$, but acts as the identity on the rest of $I_\mu$,
\item[$\circ$] $\phi^{\a,\b}_\mu$ acts as the identity on $R \cap I_\mu$, and
\item[$\circ$] $\phi^{\a,\b}_\mu = \phi^{\g,\b}_\mu \circ \phi^{\a,\g}_\mu$ for any $\g < \lambda$.
\end{itemize}
For each $\a,\b < \lambda$, let $\phi^{\a,\b}$ denote the product map $\bigotimes_{\mu \in C}\phi^{\a,\b}_\mu$ (that is, the map defined by setting $\phi^{\a,\b}(i) = \phi^{\a,\b}_\mu(i)$ whenever $i \in I_\mu$).
This is an involution of $I$, and restricts to $\phi^{\a,\b}_\mu$ on each $I_\mu$. In particular, Lemma~\ref{lem:permutations} applies, and $\bar \phi^{\a,\b} \rest \PP_\k$ is an automorphism of $\PP_\k$ for each $\a,\b < \lambda$.

\begin{lemma}
There are at most $\k$ nice $\PP_\k$-names $\dot x$ for subsets of $\w$ with the property that $\mathrm{hdom}(\dot x) \sub D^0$.
\end{lemma}
\begin{proof}
For each $\a \leq \k$, let $\PP_\a \rest D^0 = \set{p \in \PP_\a}{\mathrm{hdom}(p) \sub D^0}$.
We prove by transfinite induction on $\a$ that for all $\a \leq \k$, $\card{\PP_\a \rest D^0} \leq \k$ and there are $\leq\! \k$ nice $\PP_\a \rest D^0$-names for subsets of $\w$. Note that a nice $\PP_\a \rest D^0$-name for a subset of $\w$ is the same thing as a nice $\PP_\a$-name $\dot x$ for a subset of $\w$ with the property that $\mathrm{hdom}(\dot x) \sub D^0$.

For the base case, $\PP_0 \rest D^0$ is just the set of finite partial functions $D^0 \to 2$, so $\card{\PP_\a \rest D^0} = \aleph_0 \leq \k$ because $D^0$ is countable. Because $\PP_0 \rest D^0$ has the ccc, there are $\aleph_0$ nice $\PP_0 \rest D^0$-names for subsets of $\w$.

For the other cases, fix $\a \leq \k$, and suppose $\card{\PP_\xi \rest D^0} \leq \k$ and there are $\leq\! \k$ nice $\PP_\xi \rest D^0$-names for subsets of $\w$, for all $\xi < \a$. 
A condition in $\PP^\a \rest D^0$ is defined in the same way as a condition in $\PP^\a$, except that we restrict the $\dot b$ in that definition to members of $\PP_\g \rest D^0$. Using the inductive hypothesis on $\PP_\g \rest D^0$-names, it follows that $\card{\PP_\a \rest D^0} \leq \k$.
Each nice $\PP_\a \rest D^0$-name $\dot x$ for a subset of $\w$ has the form $\dot x = \set{\< n,p \>}{p \in \A_n}$, where each $\A_n$ is an antichain in $\PP_\a \rest D^0$.
Because $\PP_\a$ has the ccc, each $\A_n$ is countable.
Thus there are $\k^{\aleph_0} = \k$ names $\dot x$ with this form.
\end{proof}

In particular, there are $\k$ nice $\PP_\k$-names $\dot x$ for subsets of $\w$ having the property that $\mathrm{hdom}(\dot x) \sub D^0$.
But for each $\a \in \A_1$, $\bar \phi^{\a,0}(\dot B_\a)$ is just such a name.
By the pigeonhole principle, and the fact that $\k < \nu$, this implies that there is some $\A_2 \sub \A_1$ with $\card{\A_2} = \nu$ such that
$\bar \phi^{\a,0}(\dot B_\a) = \bar \phi^{\b,0}(\dot B_\b)$ whenever $\a,\b \in \A_2$. 
Reindexing the $\dot B_\a$'s if necessary, we may (and do) assume $0 \in \A_2$.

We now proceed to define a new $\PP_\k$-name $\dot B_\lambda$ for a subset of $\w$. 

First, for all $\mu \in C$ with $\mu < \lambda$, let $D^\lambda_\mu = R \cap I_\mu$. (Recall that $D^\a_\mu \cap I_\mu \sub R$ whenever $\mu \in C$ and $\mu < \lambda$ and $\a \in \A_1$. Thus, in this case, $D^\lambda_\mu \cap D^\a_\mu = R \cap I_\mu$ for all $\a \in \A_1 \supseteq \A_2$.)
Next, for each $\mu \in C$ with $\mu > \lambda$, let $D^\lambda_\mu$ be a countable subset of $I_\mu$ with $D^\lambda_\mu \cap \bigcup \set{D^\a_\mu}{\a < \lambda} = R \cap I_\mu$, and with $|D^\lambda_\mu \setminus R| = |D^0_\mu \setminus R|$.
Some such set exists because $\card{\bigcup \set{D^\a_\mu}{\a < \lambda}} = \lambda < \mu = |I_\mu|$, so we may take $D^\lambda_\mu$ to be any subset (of the appropriate size) of $I_\mu \setminus \bigcup \set{D^\a_\mu}{\a < \lambda}$, together with $R \cap I_\mu$.

For each $\mu \in C$, fix a bijection $\phi^\lambda_\mu: D^\lambda_\mu \to \w$ such that $\phi^\lambda_\mu \rest R = \phi^0_\mu \rest R$. 
For each $\a < \lambda$ and $\mu \in C$, let $\phi^{\a,\lambda}_\mu$ and $\phi^{\lambda,\a}_\mu$ be the involutions of $I_\mu$ defined just like the $\phi^{\a,\b}_\mu$ above, but using $\phi^\lambda_\mu$ in place of $\phi^\b_\mu$.
This naturally extends the system of involutions described above: for any $\a,\b \leq \lambda$,
\begin{itemize}
\item[$\circ$] $\phi^{\a,\b}_\mu$ maps $D^\a_\mu$ onto $D^\b_\mu$ and $D^\b_\mu$ onto $D^\a_\mu$, but acts as the identity on the rest of $I_\mu$,
\item[$\circ$] each $\phi^{\a,\b}_\mu$ acts as the identity on $R \cap I_\mu$, and
\item[$\circ$] $\phi^{\a,\b}_\mu = \phi^{\g,\b}_\mu \circ \phi^{\a,\g}_\mu$ for any $\g \leq \lambda$.
\end{itemize}
For each $\a,\b \leq \lambda$, let $\phi^{\a,\b} = \bigotimes_{\mu \in C}\phi^{\a,\b}_\mu$. By Lemma~\ref{lem:permutations}, $\bar \phi^{\a,\b} \rest \PP_\k$ is an automorphism of $\PP_\k$ for each $\a,\b \leq \lambda$.

Because $\mathrm{hdom}_\mu(q) \sub R$ for all $\mu$ and $\phi^{0,\lambda}$ fixes $R$, $\bar \phi^{0,\lambda}(q) = q$. Recall 
$$q \forces \text{ the evaluation $B_0$ of $\dot B_0$ codes a nonempty Borel subset of $\R$}.$$
Applying standard facts about automorphisms of forcing posets,
$$\bar \phi^{0,\lambda}(q) \forces \text{ the evaluation of $\bar \phi^{0,\lambda}(\dot B_\lambda)$ codes a nonempty Borel subset of $\R$}.$$
Let $B_\lambda$ denote the evaluation of $\dot B_\lambda$ in $V^{\PP_\k}$, and let $\tilde B_\lambda$ be the Borel set that it codes.
So, in particular, the displayed statement above implies that $q = \bar \phi^{0,\lambda}(q) \forces \tilde B_\lambda \neq \0$.
To show $q \forces \bigcup_{\a < \lambda}\tilde B_\a \neq \R$, it now suffices to show $q \forces$ $\tilde B_\lambda \cap \tilde B_\b = \0$ for all $\b < \lambda$.

Fix $\b < \lambda$. 
By our choice of the sets $D^\lambda_\mu$, we have $D^\lambda_\mu \cap D^\b_\mu \sub R$ for all $\mu \in C$, and therefore $D^\lambda \cap D^\b \sub R$. 
Because $\set{D^\a}{\a \in \A_2}$ is a $\Delta$-system of size $\nu$, and because $|D^\b| \leq |C| \cdot \aleph_0 < \k < \nu$, there is some $\a \in \A_2$ with $D^\a \cap D^\b \sub R$ (regardless of whether $\b \in \A_2$).
Fix some such $\a$.

Note that $\phi^{\a,\lambda}$ sends $D^\a$ to $D^\lambda$, while acting as the identity on $D^\b$. 
And because $0,\a \in \A_2$, our choice of $\A_2$ implies $\bar \phi^{0,\a}(\dot B_0) = \dot B_\a$.
It follows that 
$$\bar \phi^{\a,\lambda}(\dot B_\a) = \bar \phi^{\a,\lambda} \circ \bar \phi^{0,\a}(\dot B_0) = \bar \phi^{0,\lambda}(\dot B_0) = \dot B_\lambda,$$
while $\bar \phi^{\a,\lambda}(\dot B_\b) = \dot B_\b$. 
But 
$$q \forces \ B_\a \text{ and } B_\b \text{ code disjoint Borel sets}.$$
Applying the automorphism $\bar \phi^{\a,\lambda}$ of $\PP_\k$,
$$\bar \phi^{\a,\lambda}(q) = q \forces \ B_\lambda \text{ and } B_\b \text{ code disjoint Borel sets}.$$
Because $\b$ was arbitrary, this shows that $q \forces$ $\tilde B_\lambda \cap \tilde B_\b = \0$ for all $\b < \lambda$.
In other words, $q \forces \bigcup_{\a < \lambda}\tilde B_\a \neq \R$.
\end{proof}

Note that Theorem~\ref{thm:main} only gives us models in which $\ser$ is regular and $|\Spec| \leq \ser$. Both of these are merely artifacts of the proof: neither need be true of $\Spec$ in general. As mentioned near the end of Section 2, it is possible that $\dom = \continuum = \k$ for some singular cardinal $\k$ of uncountable cofinality, and this makes $\ser = \k$ also. And of course, Theorem~\ref{thm:BM} shows it is possible to have $|\Spec| > \ser$.

Theorem~\ref{thm:main} also cannot produce a model in which $\Spec \cap \k$ is unbounded for a regular limit cardinal $\k$. Theorem~\ref{thm:BM} implies that this is possible, but in such a case automatically gives $\k \in \Spec$. This suggests the following question:

\begin{question}
Is $\mathfrak{sp}(\text{\emph{\small closed}})$ or $\mathfrak{sp}(\text{\emph{\small Borel}})$ closed under regular limits?
\end{question}

We also do not know whether our final condition on $C$ is an artifact of the proof.

\begin{question}
Is it consistent that $\aleph_\w \in \Spec$ but $\aleph_{\w+1} \notin \Spec$? What about $\bspec$?
\end{question}

For any pointclass $\Gamma$, define
$$\specg = \set{|\P| > \aleph_0}{\vphantom{2^i}\P \text{ is a partition of } \R \text{ into sets in } \Gamma }\!.$$
To strengthen Theorem~\ref{thm:main}, we could replace $\bspec$ in the conclusion of the theorem with $\specg$ for some pointclass $\Gamma$ strictly containing the Borel sets.
The ultimate result in this direction would be to take $\Gamma = \text{OD}(\R)$, since this is essentially the largest pointclass of interest from a descriptive point of view.


While we have chosen to focus on $\bspec$ and $\Spec$ until now, let us observe that minor modifications to the proof of Theorem~\ref{thm:main} will give a proof of the stronger version, where $\bspec$ is replaced by $\specod$ in the conclusion.
Consequently, Corollaries~\ref{cor:1} and \ref{cor:2} also can be strengthened by replacing $\bspec$ with $\specod$. 
This has the interesting consequence that large cardinals do not imply the existence of a partition of $\R$ into $\aleph_2$ projective sets. This consequence is already known, by the work of Blass \cite{Blass} mentioned at the start of this section.

We close with three more open questions.

\begin{question}
Is $\mathfrak{sp}(\text{\emph{\small Borel}}) = \mathfrak{sp}(\text{\emph{\footnotesize OD($\R$)}})$?
\end{question}

\begin{question}
Given $\a < \w_1$, is it consistent that $\mathfrak{sp}(\text{\emph{\small $\mathbf{\Pi}^0_\a$}}) \neq \mathfrak{sp}(\text{\emph{\small $\mathbf{\Pi}^0_{\a+1}$}})$?
\end{question}

The answer to this question is currently known for $\a = 1,2$. 
For $\a=1$, Miller proved the consistency of $\mathfrak{sp}(\text{\emph{\small $\mathbf{\Pi}^0_1$}}) \neq \mathfrak{sp}(\text{\emph{\small $\mathbf{\Pi}^0_2$}})$ (i.e., $\Spec \neq \mathfrak{sp}(\text{\small $G_\dlt$})$) by constructing a model with $\aleph_1 = \mathrm{cov}(\mathcal M) < \ser$ \cite[Theorem 5]{Miller}. This shows $\Spec \neq \mathfrak{sp}(\text{\small $G_\dlt$})$, because if $\aleph_1 = \mathrm{cov}(\mathcal M)$, then $\aleph_1 \in \mathfrak{sp}(\text{\small $G_\dlt$})$. (Proof: if $\set{F_\a}{\a < \w_1}$ is a collection of closed nowhere dense sets covering $\R$, then $\{ F_\a \setminus \bigcup_{\xi < \a} :\, \a < \w_1 \}$ is a partition of $\R$ into $G_\dlt$ sets.)
For $\a = 2$, the consistency of $\mathfrak{sp}(\text{\emph{\small $\mathbf{\Pi}^0_2$}}) \neq \mathfrak{sp}(\text{\emph{\small $\mathbf{\Pi}^0_3$}})$ (i.e., $\mathfrak{sp}(\text{\emph{\small $G_\dlt$}}) \neq \mathfrak{sp}(\text{\emph{\small $F_{\s \dlt}$}})$) follows from Hausdorff's theorem that $\aleph_1 \in \mathfrak{sp}(\text{\emph{\small $\mathbf{\Pi}^0_3$}})$, together with a result of Fremlin and Shelah \cite{Fremlin&Shelah} stating that $\min \!\big( \mathfrak{sp}(\text{\emph{\small $\mathbf{\Pi}^0_2$}}) \big) \geq \mathrm{cov}(\mathcal M)$.

\begin{question}
Is it consistent that $\mathfrak{sp}(\text{\emph{\small closed}}) \neq \mathfrak{sp}(\text{\emph{\small $G_\dlt$}})$ and $\mathfrak{cov}(\mathcal M) > \aleph_1$?
\end{question}

%
%
%



\end{document}